\documentclass[12pt,oneside]{amsart}
\usepackage[a4paper]{geometry}  % ,margin=1in

\usepackage{color}

\usepackage{amssymb}

\usepackage[matrix,arrow]{xy}

\usepackage{graphicx}

\usepackage[english]{babel}

\usepackage{hyperref}
\usepackage{mathptmx}

\theoremstyle{plain}
 \newtheorem{theorem}{Theorem}[section]
\theoremstyle{definition}
 \newtheorem{definition}[theorem]{Definition}

\theoremstyle{plain}
 \newtheorem{proposition}[theorem]{Proposition}
 
 \newtheorem{lemma}[theorem]{Lemma} %%Delete [theorem] to re-start numbering
\theoremstyle{remark}
 \newtheorem{remark}[theorem]{Remark}
 \newtheorem{example}[theorem]{Example}

\def\Z{\mathbb{Z}}
\def\F{\mathbb{F}}
\DeclareMathOperator{\girth}{girth}
\DeclareMathOperator{\id}{id}
\DeclareMathOperator{\diam}{diam}

\newcommand{\Hawaii}{Hawai\kern.05em`\kern.05em\relax i}
\newcommand{\Manoa}{M\=anoa}

\definecolor{m}{rgb}{1,0.1,1}

\title{Coarse non-amenability and coarse embeddings}


\author{Goulnara Arzhantseva}
\address{University of Vienna,
Faculty of Mathematics, Nordbergstra${\ss}$e 15, 1090 Wien, Austria}
\email{goulnara.arzhantseva@univie.ac.at}

\author{Erik Guentner}
\address{University of \Hawaii~at \Manoa,
Department of Mathematics,
2565 McCarthy Mall, Honolulu, HI 96822, USA}
\email{erik@math.hawaii.edu}

\author{J\'an \v Spakula}
\address{Mathematisches Institut, Universit\"at M\"unster, Einsteinstr.\ 62, 48149 M\"unster, Germany}
\email{jan.spakula@uni-muenster.de}

\thanks{The first author was partially supported by the ERC grant ANALYTIC no. 259527, the Swiss NSF
Sinergia grant CRSI22 130435, and by the CNRS, UMR 6632. The second
author was partially supported by NSF grant DMS-0349367. The third
author was supported by the Deutsche Forschungsgemeinschaft (SFB
878).}

\begin{document}

\begin{abstract}
  We construct  the first example of a
    coarsely non-amenable (= without Guoliang Yu's property A)
  metric space with bounded geometry   which coarsely embeds into a Hilbert space.
\end{abstract}

\maketitle

\section{Introduction}

The purpose of this paper is to prove the following theorem:

\begin{theorem}\label{thm:main}
There exists a uniformly discrete metric space with bounded geometry, which coarsely embeds into a Hilbert space, but does
not have property A.
\end{theorem}

The concept of coarse embedding was introduced by Gromov~\cite[p. 218]{Gr:asinv}
in relation to the Novikov conjecture (1965) on the homotopy invariance of higher signatures for closed manifolds.

\begin{definition} A metric space $X$
is said to be \emph{coarsely embeddable}  into a Hilbert space $\mathcal{H}$ if there exists a map
$f\colon X\to\mathcal{H}$ such that for any $x_n, y_n\in X$, $n\in \mathbb{N}$,
$$
\hbox{dist} (x_n,y_n)\to\infty\ \ \Longleftrightarrow\ \
\|f(x_n)-f(y_n)\|_{\mathcal{H}}\to \infty.
$$
\end{definition}

Yu established the coarse Baum-Connes conjecture (1995) in topology
 for every coarsely embeddable discrete space with bounded geometry~\cite[Theorem 1.1]{Yu}.
This implies the Novikov conjecture  for all closed manifolds whose
fundamental group, viewed with the word length metric, coarsely embeds into a Hilbert space.
The result confirmed Gromov's intuition and sparked an intense study of groups and
metric spaces which are coarsely embeddable into Hilbert space.

Coarse embeddability, which is a geometric property by nature,
shares many features with Property A, a weak form of amenability, introduced
by Yu using the following F\o lner-type condition.

\begin{definition} A discrete metric space $X$ has \emph{Property A} if for every $\varepsilon > 0$ and every $R > 0$ there is a
family $\{A_x\}_{x\in X}$ of finite subsets of $X\times \mathbb{N}$ and a number $S > 0$ such that
\begin{itemize}
  \item $\displaystyle{\frac{\vert A_x\bigtriangleup A_y\vert}
{\vert A_x \cap A_y\vert}< \varepsilon}$ whenever $d(x, y)\leqslant R$,
  \item $A_x \subseteq B(x, S) \times \mathbb{N}$ for every $x\in X$.
\end{itemize}
\end{definition}

Just as with amenability, Property A has a large number of significant applications, see the survey~\cite{willett}.
For a countable discrete group $\Gamma$,
it is equivalent to the existence of a topological amenable action of $\Gamma$ on a
compact Hausdorff space and to the $C^*$-exactness of the reduced  $C^*$-algebra  $C_r^*(\Gamma)$~\cite{Ozawa,HR}.

Property A implies coarse embeddability~\cite[Th.2.2]{Yu}. Moreover,
discrete spaces with property A provide the largest known class of
spaces admitting such an embedding. Conversely, whether or not the
existence of a coarse embedding of a \emph{bounded geometry\/}
discrete metric space guarantees property A is a crucial open
problem that has attracted much research in the area. (In the case
of unbounded geometry a counterexample was constructed by Nowak
\cite{nowak}.)

Our main result, Theorem~\ref{thm:main}, yields a negative answer to
this problem. Here is the construction of our counterexample.
 For a discrete group $\Gamma$, we denote by $\Gamma^{(2)}$
its normal (in fact characteristic) subgroup generated by all the
squares of elements of $\Gamma$. Let {$\F_2$ be the free group
of rank two. We define inductively a sequence of
normal subgroups of $\F_2$ by letting $\Gamma_0=\F_2$ and
$\Gamma_n=\Gamma_{n-1}^{(2)}$, $n\geqslant 1$. We denote by $X_n$ the Cayley
graph of $\F_2/\Gamma_n$ with respect to the image of the canonical generators of $\F_2$.}

\begin{theorem}\label{thm:box-space}
The space $X=\bigsqcup_{n=0}^{\infty} X_n,$ (the box space of\, $\F_2$  {associated with
the family $(\Gamma_n)_{n\geqslant 1}$)} is coarsely embeddable into a Hilbert space, but does not have property A.
\end{theorem}

 In outline the proof goes as follows. We describe the Cayley graphs $X_n$ in a graph-theoretical
way as a tower of successive $\Z/2$-homology covers, starting with the ``figure eight'' graph.
Next, we define a wall structure on each $X_n$, which gives rise to a wall metric on the graphs $X_n$.
 We show that these graphs, endowed with the wall metric, coarsely embed into a Hilbert space (uniformly).
Their natural graph metric does not coincide with the wall metric. Nevertheless,
we prove that on the small scale it does, which is then enough to conclude that the
two metrics are coarsely equivalent. Consequently, $X$ is also coarsely embeddable.
On the other hand, it is a fact that $\bigcap_{n\geqslant1}\Gamma_n=\{1\}$, so if $X$ would have property A,  {the free group} $\F_2$ would
be amenable, which is certainly not the case.

The structure of the paper is as follows: in Section \ref{sec:covers} we review the construction of graph coverings,
in particular $\Z/2$-homology coverings, and list their properties. In Section \ref{sec:walls}
we define a wall structure on the $\Z/2$-homology cover of a graph and investigate its
relation to the graph metric. In Section \ref{sec:results} we prove Theorem \ref{thm:box-space}. In the last Section \ref{sec:apps} we  {discuss}
applications to C*-algebras,
namely  {to variants of the uniform} Roe $C^*$-algebras.

\section{Constructing graph coverings}\label{sec:covers}

\subsection{Graph terminology}
For a graph $G$, that is, an 1-dimensional simplicial complex, we denote by $V(G)$ the set
of vertices and by $E(G)$ the set of edges of $G$. Our graphs will be unoriented, with a couple
of exceptions which are specified in the text later. We shall talk about an edge $e$ being \emph{between} two vertices $x$ and $y$.
We allow edges with endpoints being the same (sometimes in the literature called ``loops'', but we shall not call them that as not to
confuse them with loops defined below) and multiple edges between pairs of vertices.

A \emph{path} $p$ in $G$ will mean a sequence $(v_0,e_1,v_1,e_2,v_2,\dots,v_n)$, such that $v_i\in V(G)$, $e_i\in E(G)$ and
each $e_i$ is an edge between $v_{i-1}$ and $v_i$. We denote $E(p)=(e_1,e_2,\dots,e_n)$ the sequence of edges of $p$. We also
write $\ell(p)=n$ for the \emph{length} of the path $p$.

A path $l=(v_0,e_1,v_1,e_2,v_2,\dots,v_n)$ is called a \emph{loop}, if $v_0=v_n$ and $n>0$ (we don't
consider one vertex as a loop). The loop $l$ is called \emph{simple}, if there are no repeated vertices or edges, i.e.\ $v_0,\dots,v_{n-1}$ and $e_1,\dots,e_n$ are all different. The \emph{girth} of a graph $G$ is the
length of a shortest simple loop.

When we talk about topological properties of $G$ (e.g.\ connectedness), we consider $G$ as a realization of an $1$-dimensional
simplicial complex determined by the structure of $G$.

\subsection{Graph coverings: general case}

The notion of graph covering is essentially a restriction to the
case of graphs of the general topological notion of covering. For the sake of completeness,
we recall this classical construction (see, for example,~\cite[Ch. 2]{stillwell} or~\cite[Ch. 1]{hatcher})
and subsequently specialize it to the $\Z/2$-homology covering situation.

We start with the following data: a finite connected graph $G$
and a surjective homomorphism $\rho: \pi_1(G)\twoheadrightarrow K$ from the fundamental group of $G$
onto a finite group $K$.

We construct the corresponding \emph{covering graph} $\widetilde G$
of $G$ as follows. Choose a spanning tree $T$ of $G$, that is, a
contractible subgraph $T$ of $G$ with $V(T)=V(G).$  Note that then
necessarily $|E(T)|=|V(T)|-1$. The remaining edges of $G$, denoted
by $S=E(G)\setminus E(T)$, identify $\pi_1(G)$ with the free group
$\F(S)$ on $S$: contracting $T$ to a point gives a homotopy
equivalence of $G$ with a bouquet of $|S|$ circles. With this
identification, we consider the domain of $\rho$ to be $\F(S)$, that
is, $\rho\colon\F(S)\twoheadrightarrow K$. Finally, we choose an
orientation on each edge in $S$, we shall use the self-explanatory
notation $x\buildrel{e}\over\rightarrow y$.
Then the covering graph $\widetilde G$ is defined by:
\begin{itemize}
\item $V(\widetilde G) = V(G)\times K$,
\item $E(\widetilde G) = E(G)\times K$,
 where the incidence of an edge $(e,k)\in E(\tilde G)$ is as follows:
\begin{itemize}
\item if $e\in E(T)$ is an edge between $x$ and $y$, then $(e,k)$ connects $(x,k)$ with $(y,k)$,
\item if $e\in S$ and $x\buildrel{e}\over\rightarrow y$, then $(e,k)$ connects $(x,k)$ with $(y,\rho(e)k)$.
\end{itemize}
\end{itemize}
The \emph{covering projection} $\pi\colon\widetilde G\to G$ is the obvious one $V(\widetilde G)\to V(G)$, $E(\widetilde G)\to E(G)$.
We will refer to the full subgraphs of $\widetilde G$ spanned by the vertices $V(G)\times\{k\}$ as \emph{clouds}.

\begin{figure}[h]
\begin{center}
\ifpdf
\includegraphics{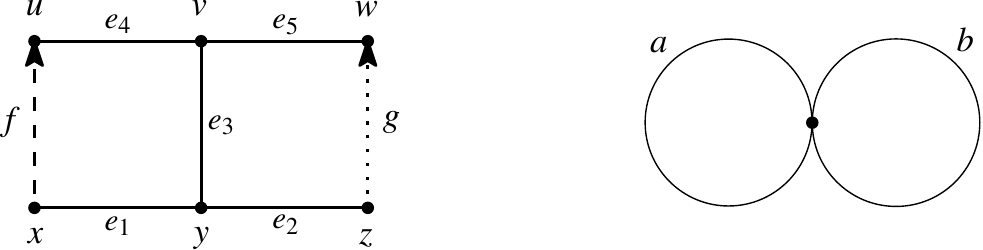}
\else\fi
\end{center}
\caption{The graph $H$ from Example \ref{ex:H-graph} (left); the ``figure eight'' graph (right).}
\label{fig:H-graph}
\end{figure}

\begin{example}\label{ex:H-graph}
Let $H$ be the graph given by Figure \ref{fig:H-graph}.
We choose $S=\{f,g\}$, the dashed and the dotted edge; the remaining edges form the spanning tree $T$.
Let us consider the surjective homomorphism $\rho:\F(S)\twoheadrightarrow \Z/2\times \Z/2$ defined by assigning $\rho(f)=(1,0)$ and $\rho(g)=(0,1)$.
Observe that this is in fact an example of a $\Z/2$-homology covering as described below.

The covering graph $\widetilde H$, see Figure \ref{fig:tildeH},
has four clouds $\text{\framebox{$\bf (i,j)$}}=V(G)\times\{(i,j)\}$, where $i,j\in\{0,1\}$.
Within these clouds, one sees copies of the spanning tree $T$. The edges from $S$ connect different clouds according to the action of $\rho$.
Note that collapsing clouds to points yields the Cayley graph of $K=\Z/2\times\Z/2$ with respect to the generating set $\rho(S)$.
\end{example}

\begin{figure}[h]
\begin{center}
\ifpdf
\includegraphics{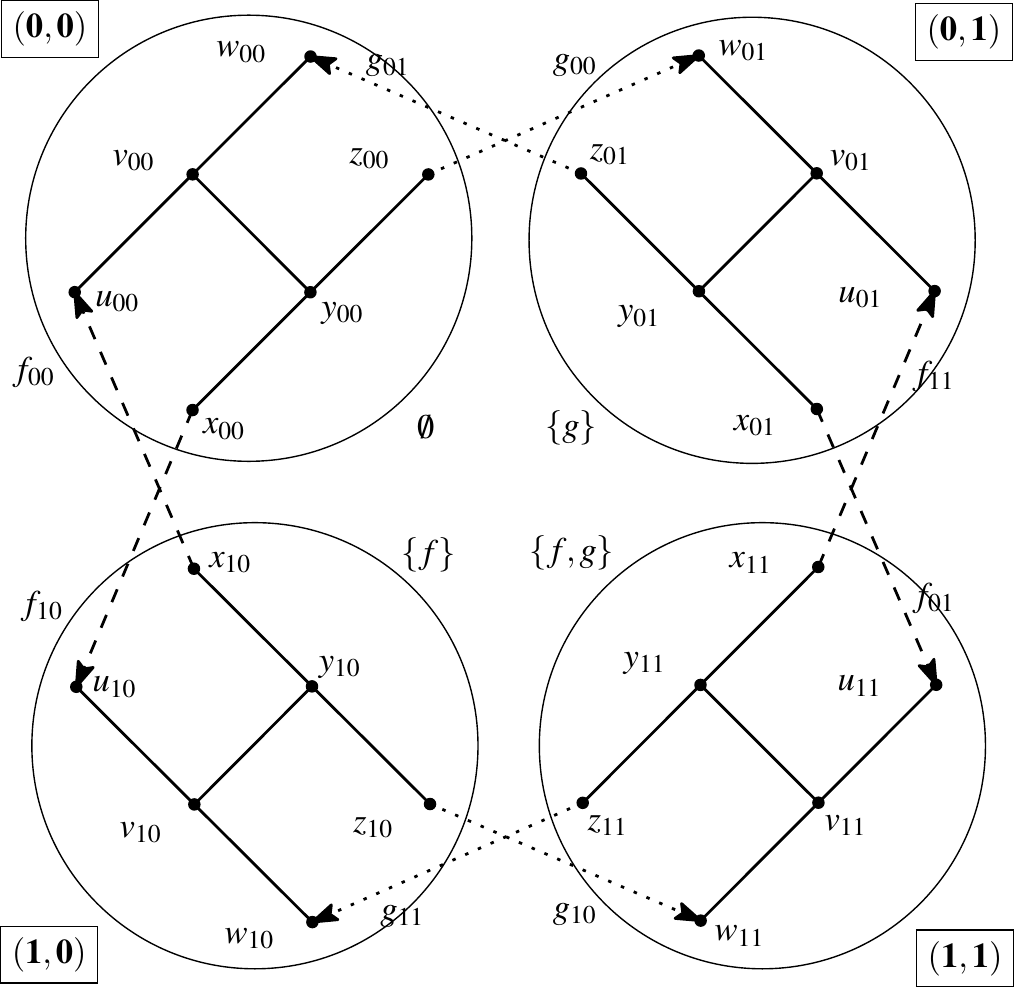}
\else\fi
\end{center}
\caption{The graph $\widetilde H$ from Example \ref{ex:H-graph}.}
\label{fig:tildeH}
\end{figure}

Let us summarize basic properties of this construction, see, for example,~\cite[Ch. 2]{stillwell}.
\begin{proposition}
Let $G$ be a finite connected graph, $\rho: \pi_1(G)\twoheadrightarrow K$ be a surjective homomorphism, and
let $\widetilde G$ be the covering graph constructed above. Then the following hold.
\begin{itemize}
\item The covering projection $\pi\colon \widetilde G\to G$ enjoys the unique path lifting property and the homotopy lifting property.
\item $\widetilde G$ does not depend on the choices made, i.e.\ on $T$ and on the orientation of edges in $S$.
It is unique, up to graph isomorphism commuting with the projections.
\item There is a short exact sequence $1\to \pi_1(\widetilde G)\buildrel\pi_*\over\hookrightarrow \pi_1(G) \buildrel\rho\over\twoheadrightarrow K\to 1$
and the  group of deck transformations of the covering is $K$.
\end{itemize}
\end{proposition}

Observe that if $G$ is the ``figure eight'', see Figure~\ref{fig:H-graph},
and we identify $\pi_1(G)$ with the free group $\F_2$ generated by $a$ and $b$,
then $\widetilde G$ is the Cayley graph of $K$ with respect to the generators $\rho(a)$ and $\rho(b)$.

The following classical result follows from the fundamental theorem of
Galois theory for graph coverings, see, for example,~\cite[Section 1.3]{hatcher}.

\begin{proposition}[Transitivity of covers]\label{prop:transitivity}
Let $G$ be a finite connected graph and $L=\pi_1(G)$ denote its fundamental group. Let $M$ and $N$ be two normal subgroups of $L$
satisfying $M\subset N$. Let us denote by
\begin{itemize}
\item $\pi^M:G^M\to G$  and $\pi^N:G^N\to G$ the coverings of $G$ corresponding to the surjective
homomorphisms $\rho^M\colon L\twoheadrightarrow L/M$ and
 $\rho^N\colon L\twoheadrightarrow L/N$, respectively;
\item $\pi^M_N:G^M_N\to G^N$ the covering of $G^N$ corresponding to $\rho^M_N\colon N\twoheadrightarrow N/M$.
(Observe that $\pi_1(G^N)\cong N$ by the basic properties of the construction.)
\end{itemize}
Then, $\pi^M:G^M\to G$ is isomorphic to $\pi^M_N\circ\pi^N:G^M_N\to G$ as graph coverings of $G$.
$$
\xymatrix{G^M\ar[dd]_{\pi^M}\ar[r]^{\simeq}& G^M_N\ar[d]^{\pi^M_N}\\
 & G^N\ar[ld]^{\pi^N}\\
 G &
}
$$
\end{proposition}

\subsection{Graph coverings: $\Z/2$-homology case}\label{subsec:z2}

Let us now specialize to the $\Z/2$-homo\-logy coverings. Given a
finite graph $G$, there is always a surjective homomorphism $\rho$
from the fundamental group $\pi_1(G)$ to the first $\Z/2$-homology
group of $G$. Indeed, it is the quotient map
$\rho:\pi_1(G)\twoheadrightarrow \pi_1(G)/\pi_1(G)^{(2)}$. Recall
that  $\Gamma^{(2)}$ denotes the normal subgroup generated by all
the squares of elements in a group $\Gamma$. The resulting quotient
is always the $n$-fold direct sum of $\Z/2$'s, where $n$ is the rank
of the free group $\pi_1(G)$.

It is convenient to label the elements of the
group $K\cong (\Z/2)^{|S|}$ by the elements of the power set $\mathcal{P}(S)$ of
$S$; thinking of factors of $(\Z/2)^{|S|}$ indexed by
elements of $S$, a subset $\tau\subset S$ denotes the element of $(\Z/2)^{|S|}$
having $1$ at places indexed by elements from $\tau$ and $0$ at the
remaining places. In this description, an edge from $S$ acts on $\mathcal{P}(S)$
by ``symmetric difference'': for $\tau\subset S$ and $e\in S$, we have
$\rho(e)\tau=\tau\triangle\{e\}$.

Summarizing, given a finite connected graph $G$, choose a spanning tree $T$ of $G$ and orient the edges in $S=E(G)\setminus E(T)$ arbitrarily.
The vertices of $\widetilde G$ will be indexed by $V(G)\times\mathcal{P}(S)$,
the edges by $E(G)\times \mathcal{P}(S)$. An edge $(e,\tau)$, where $e\in E(G)$ between $x,y\in V(G)$, $\tau\subset S$, connects
\begin{itemize}
\item $(x,\tau)$ and $(y,\tau)$ if $e\not\in S$,
\item $(x,\tau)$ and $(y,\tau\triangle\{e\})$ if $e\in S$ and $x\buildrel e\over\rightarrow y$.
\end{itemize}

Observe that if we collapse clouds of $\widetilde G$ to points, we obtain the Cayley graph of $(\Z/2)^{|S|}$
with respect to its natural set of generators; that is, an $|S|$-dimensional cube.

\begin{example}\label{ex:H-as-Z2}
The graph $\widetilde H$ from Example \ref{ex:H-graph} is an instance of a $\Z/2$-homology cover;
According to our specific labelling, the elements of $\Z/2\times\Z/2$ are encoded by
the subsets of $\{f,g\}$. Here this encoding is as
follows: \framebox{$\bf (0,0)$}$\to \emptyset$, \framebox{$\bf (1,0)$}$\to\{f\}$, \framebox{$\bf (0,1)$}$\to\{g\}$, \framebox{$\bf (1,1)$}$\to\{f,g\}$.
\end{example}

\section{Walls}\label{sec:walls}

In this section, we first briefly recall the notions of walls and
wall structures on a graph. Our formulation lies between the
classical concept of a graph cut and a more involved notion of a
wall structure associated to certain polyhedral complexes
\cite{haglund-paulin}. Next, we construct a wall structure on
$\Z/2$-homology covers from the previous section, and establish a
relationship between the wall metric and the graph metric in this
context. Finally, in Proposition \ref{prop:same-balls}, we show that
the two metrics agree on small scale (depending on the girth of the
graphs involved).

\begin{definition}
Let $G$ be a finite graph. A \emph{wall} (or a \emph{cut}) on $G$ is a set of edges of $G$, such that removing them from $G$ separates $G$
into exactly two connected components. We shall refer to the components as to \emph{half-spaces} associated to the wall.

A \emph{wall structure} $W$ on a graph $G$ is a set of walls on $G$, such that each edge is contained in exactly one wall from $W$.

We say that a wall $w$ \emph{separates} two vertices of $G$, if they reside in different half-spaces associated to $w$.

Given a wall structure $W$ on $G$, we denote by $d_W(x,y)$ the
number of walls in $W$ separating the vertices $x$ and $y$ of $G$.
Then $d_W$ is a pseudo-metric\footnote{It turns out later in this
section that in our case it is in fact a metric. We call $d_W$ a
metric right away.} on $G$, which we call the \emph{wall metric}
associated to the wall structure $W$.
\end{definition}

Note that a graph need not admit any wall structure.

 Let us now specialize to the situation of the previous section.
\begin{definition}\label{def:z2-pair}
A pair $(\widetilde G,G)$ is said to be a \emph{$\Z/2$-pair}, if  $G$ is a
finite $2$-connected\footnote{A graph is \emph{2-connected} if every edge lies on a simple loop, or
equivalently, if removing any single edge does not disconnect the graph.} graph and $\widetilde G$
is its $\Z/2$-homology cover.
\end{definition}

Recall that $T$ denotes a chosen spanning tree in a finite graph $G$ and $S=E(G)\setminus E(T)$.
For an edge $e\in E(G)$ and the covering map $\pi:\widetilde G\to G$, we
denote $w_e=\pi^{-1}(e)\subset E(\widetilde G)$ and $W=\{w_e\mid e\in E(G)\}$.

\begin{lemma}\label{lem:separating}
If $(\widetilde G,G)$ is a $\Z/2$-pair, each $w_e$ separates $\widetilde G$ into exactly two connected components.
Thus, $W$ is a wall structure on $\widetilde G$.
\end{lemma}

\begin{proof}
If $e\in S$, denote $A=\{\tau\subset S\mid e\not\in \tau\}$ and
$B=\{\tau\subset S\mid e\in \tau\}$. We claim that removing
$w_e=\pi^{-1}(e)$ from $\widetilde G$, we obtain two components,
which happen to be the full subgraphs $\widetilde A$ and $\widetilde
B$ of $\widetilde G$ spanned by the sets of vertices $V(G)\times A$
and $V(G)\times B$, respectively\footnote{Thinking of $\widetilde G$
roughly as a cube (vertices being clouds), this splitting
corresponds to choosing a coordinate direction and taking ``the
front'' and ``the back'' of the cube. The wall consists of edges
connecting front to back and vice versa.}. Indeed, if an
edge in $w_e$ connects two vertices $(\cdot,\tau)$ and
$(\cdot,\tau')$ in $V(\widetilde G)$, then $\tau\triangle
\tau'=\{e\}$, thus one of the vertices is in $V(\widetilde A)$ and
the other in $V(\widetilde B)$. Conversely, in a similar way, any
edge between $\widetilde A$ and $\widetilde B$ has to be in $w_e$.
Observe that $\widetilde A$ is connected: we can move within the
clouds along the copies of the spanning tree $T$, and from cloud to
cloud within $\widetilde A$ because if $\tau,\tau'\subset A$, then
$\tau\triangle \tau'$ never contains $e$. The same argument works
for $\widetilde B$. Thus, the set $w_e$ is a wall whenever $e\in S$.

Since $G$ is 2-connected, given any edge $e\in E(G)$, we may
choose a spanning tree $T$ of $G$ which does not contain $e$ (removing $e$ from $G$ leaves it connected,
so choose a spanning tree there). Such a choice makes $e$ belong to the corresponding $S$, so the argument
above applies. Here we use the fact that the covering $\pi:\widetilde G\to G$ is unique,
the choices just give a different labelling.

In fact, one can avoid the preceding re-labelling trick and give a
direct argument for edges in $E(T)$. It turns out that even if $G$ is not $2$-connected,
$w_e$ for $e\in E(T)$ separates $\widetilde G$ into at least two components;
using $2$-connectedness one gets that there are at most 2 components.
We leave this as an exercise for the reader.
\end{proof}

We emphasize that $\widetilde G$ is now endowed with two metrics: $d_{\widetilde G}$, the usual graph metric and $d_W$,
the wall metric with respect to the wall structure $W$ constructed above. The crucial aspect is to compare these two metrics.
This task will occupy the rest of this section. We begin with an easy observation.

\begin{proposition}\label{prop:WlessG}
If $(\widetilde G,G)$ is a $\Z/2$-pair, we have $d_W(x,y)\leqslant d_{\widetilde G}(x,y)$ for all $x,y\in V(\widetilde G)$.
\end{proposition}

\begin{proof}
If a wall separates $x$ and $y$, any path between $x$ and $y$ necessarily uses at least one of the edges from the wall.
Since different walls are disjoint, any shortest path between $x$ and $y$ has to traverse at least $d_W(x,y)$ edges.
\end{proof}

The following notion is used to characterize $d_{\widetilde G}$ and $d_W$ in terms of paths in $G$.

\begin{definition}
Let $(\widetilde G,G)$ be a $\Z/2$-pair.
Given $x=(x_0,\tau_x)$ and $y=(y_0,\tau_y)$ from $V(\widetilde G)=V(G)\times \mathcal{P}(S)$, we say that
a path $p$ in $G$ is \emph{$(x,y)$--admissible}, if
\begin{itemize}
\item[(i)] $p$ begins at $x_0$ and ends at $y_0$,
\item[(ii)] for every $e\in S$, the parity of the number of times that $e$ appears in
$E(p)$ is equal to $\begin{cases}0&\text{if }e\not\in \tau_x\triangle \tau_y\\ 1&\text{if }e\in \tau_x\triangle \tau_y\end{cases}$.
\end{itemize}
\end{definition}

\begin{lemma}\label{lem:adm-bij}
The $(x,y)$--admissible paths in $G$ are in bijection with paths from $x$ to $y$ in $\widetilde G$. Moreover, this correspondence
preserves the path length.
\end{lemma}

\begin{proof}
Take $x,y\in V(\widetilde G)$. We can represent $x=(x_0,\tau_x)$, $y=(y_0,\tau_y)$, where $\tau_x$ and $\tau_y$ are subsets of $S$.
Let $\widetilde p$ be a path in $\widetilde G$ between $x$ and $y$. The projection $\pi(\widetilde p)$ obviously satisfies (i).

Now consider the sequence of the second coordinates of the vertices appearing in $\widetilde p$, say $(\tau_x,\tau_1,\tau_2,\dots,\tau_n,\tau_y)$.
Then each two consecutive ones are either the same (if the edge between them is in $\pi^{-1}(E(T))$, i.e.\ we remain in the same cloud),
or they differ by exactly one edge $\{e\}$ (if the edge between them is in $\pi^{-1}(e)$, $e\in S$, i.e.\ we move between the clouds).
Since we begin with $\tau_x$ and end with $\tau_y$, then the parity of the number of times that some edge from $\pi^{-1}(e)$
is used in $\widetilde p$ depends only on whether $e$ belongs to $\tau_x\triangle \tau_y$ or not. Thus (ii) follows.

Conversely, if $p=(x_0,e_1,v_1,\dots,e_k,y_0)$ is an $(x,y)$--admissible path,
it has a unique lift $\widetilde p$ in $\widetilde G$ starting at $x\in V(\widetilde G)$.
By construction of $\widetilde G$,  the sequence $E(p)=(e_1,e_2,\dots,e_k)$ determines
the sequence $(\tau_x=\tau_0,\tau_1,\dots,\tau_k)$ of the second coordinates of the vertices in $\widetilde p$ by the following
inductive rule: if $e_i\in E(T)$, then $\tau_i=\tau_{i-1}$; if $e_i\in S$ then $\tau_i=\tau_{i-1}\triangle \{e_i\}$.
Now the condition (ii) plus this rule implies that the last one, $\tau_k$, must be equal to $\tau_y$.
Hence the endpoint of $\widetilde p$ is $(y_0,\tau_y)=y$ and we are done.

The claim about preserving the lengths is obvious.
\end{proof}

\begin{example}
Referring again to Example \ref{ex:H-graph} and Figure \ref{fig:H-graph}, the path
$$p_0=(x,e_1,y,e_3,v,e_5,w,g,z,e_2,y,e_1,x)$$
 is $((x,\emptyset), (x,\{g\}))$--admissible. It is the projection under $\pi$ of the path in $\widetilde G$ given
 by the sequence of vertices $(x_{00},y_{00},v_{00},w_{00},z_{01},y_{01},x_{01})$. In fact, $p_0$ is a
 shortest $((x,\emptyset), (x,\{g\}))$--admissible path. So, using Propositions below, $d_{\widetilde G}((x,\emptyset),(x,\{g\}))=6$,
 but $d_W((x,\emptyset),(x,\{g\}))=4$. The point is that the wall $w_{e_1}$ does not separate $(x,\emptyset)$ and $(x,\{g\})$, it is
 used on $p_0$ twice.
\end{example}

The next proposition follows readily from Lemma \ref{lem:adm-bij}.

\begin{proposition}\label{prop:graph-dist}
If $(\widetilde G,G)$ is a $\Z/2$-pair, for any $x,y\in V(\widetilde G)$ we have
$$
d_{\widetilde G}(x,y)=\min\{\ell(p)\mid p\text{ is an }(x,y)\text{--admissible path in }G \}.
$$
\end{proposition}

\begin{remark}\label{rem:no-backtracks}
Note that no shortest $(x,y)$--admissible path has a backtrack (i.e.\ a subpath $(v_i,e_{i+1},v_{i+1},e_{i+2},v_{i+2})$,
which satisfies $v_i=v_{i+2}$, $e_{i+1}=e_{i+2}$). The reason is that removing such a backtrack from
the path does not change admissibility of the path, thus providing a shorter admissible path.
\end{remark}

\begin{proposition}\label{prop:wall-dist}
If $(\widetilde G,G)$ is a $\Z/2$-pair, for any $x,y\in V(\widetilde G)$, $d_W(x,y)$ is equal to the number of edges appearing an odd
number of times in any $(x,y)$--admissible path in $G$.
\end{proposition}

\begin{proof}
Pick any wall $w$ in $\widetilde G$. Label the half-spaces corresponding to $w$ by $H_-^w$ and $H_+^w$. Now if a path in $\widetilde G$
uses one of the edges of $w$, it is clear that it passes from $H_-^w$ to $H_+^w$ or vice versa.
Thus, if a wall $w$ separates $x$ and $y$, then they do not belong to the same half-space of $w$, and so
every path between them has to ``cross'' $w$ an odd number of times. Likewise, if a wall $w$ does not separate $x$ and $y$, both
vertices belong to the same half-space of $w$, and any path between them has to ``cross'' $w$
an even number of times. By ``crossing'' $w$, we mean using some edge in $w$. To summarize, $d_W(x,y)$ is
equal to the number of walls that are crossed odd number of times by any path from $x$ to $y$ in $\widetilde G$.

The proof is now finished by referring to Lemma \ref{lem:adm-bij} and remembering
that walls in $\widetilde G$ are projected via $\pi$ exactly to edges in $G$.
\end{proof}

We now use the characterizations of $d_W$ and $d_{\widetilde G}$ to show that the two metrics agree on the scale of $\girth(G)$.

\begin{proposition}\label{prop:same-balls}
If $(\widetilde G,G)$ is a $\Z/2$-pair, then for every $x,y\in V(\widetilde G)$,
$$
d_W(x,y)<\girth(G) \quad\text{if and only if}\quad d_{\widetilde G}(x,y)<\girth(G).
$$
Furthermore, if the above inequalities hold, then $d_W(x,y)=d_{\widetilde G}(x,y)$.
\end{proposition}

\begin{proof}
The implication ``$\Longleftarrow$'' is trivial by Proposition \ref{prop:WlessG}. It remains to prove the
implication ``$\Longrightarrow$'' and the ``Furthermore'' part. We rely on the following lemma.
\begin{lemma}\label{lem:simple}
Given $x,y\in V(\widetilde G)$, every shortest $(x,y)$--admissible path in $G$ either does not contain a loop, or else every edge on
any loop it contains is traversed exactly once.
\end{lemma}

\begin{figure}[h]
\begin{center}
\ifpdf
\includegraphics{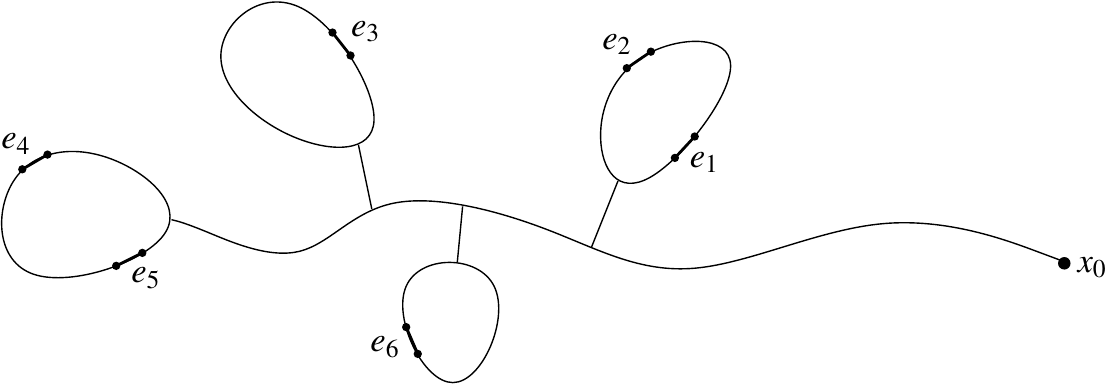}
\else\fi
\end{center}
\caption{A typical shortest $(x,y)$--admissible path; $\pi(x)=x_0=\pi(y).$}
\label{fig:spoon}
\end{figure}

Taking the lemma for granted, we finish the proof of the Proposition. Assume that $d_W(x,y)<\girth(G)$. Then
any shortest $(x,y)$--admissible path $q$ in $G$ cannot contain a loop --- otherwise the edges on it are traversed
exactly once by the above Lemma, and thus $d_W(x,y)$ is at least the length of that loop by Proposition \ref{prop:wall-dist}.
But the length of a loop in $G$ is at least $\girth(G)$, which is a contradiction. Thus we've shown that $q$ does not contain a loop.
It doesn't contain backtracks, either, see Remark \ref{rem:no-backtracks}. Thus, each edge of $q$ is traversed exactly once,
so $d_W(x,y)=\ell(q)$ by Proposition \ref{prop:wall-dist}. Now the lift $\widetilde q$ of $q$ to $\widetilde G$ is a path
between $x$ and $y$ satisfying $\ell(\widetilde q)=\ell(q)$, whence
$$
d_{\widetilde G}(x,y)\leqslant \ell(\widetilde q)=\ell(q)=d_W(x,y)<\girth(G).
$$
Note that the reverse inequality between $d_{\widetilde G}$ and $d_W$ is asserted by Proposition \ref{prop:WlessG}, so we in
fact have $d_W(x,y)=d_{\widetilde G}(x,y)$. This proves also the ``Furthermore'' part.
\end{proof}

\begin{figure}[h]
\begin{center}
\ifpdf
\includegraphics{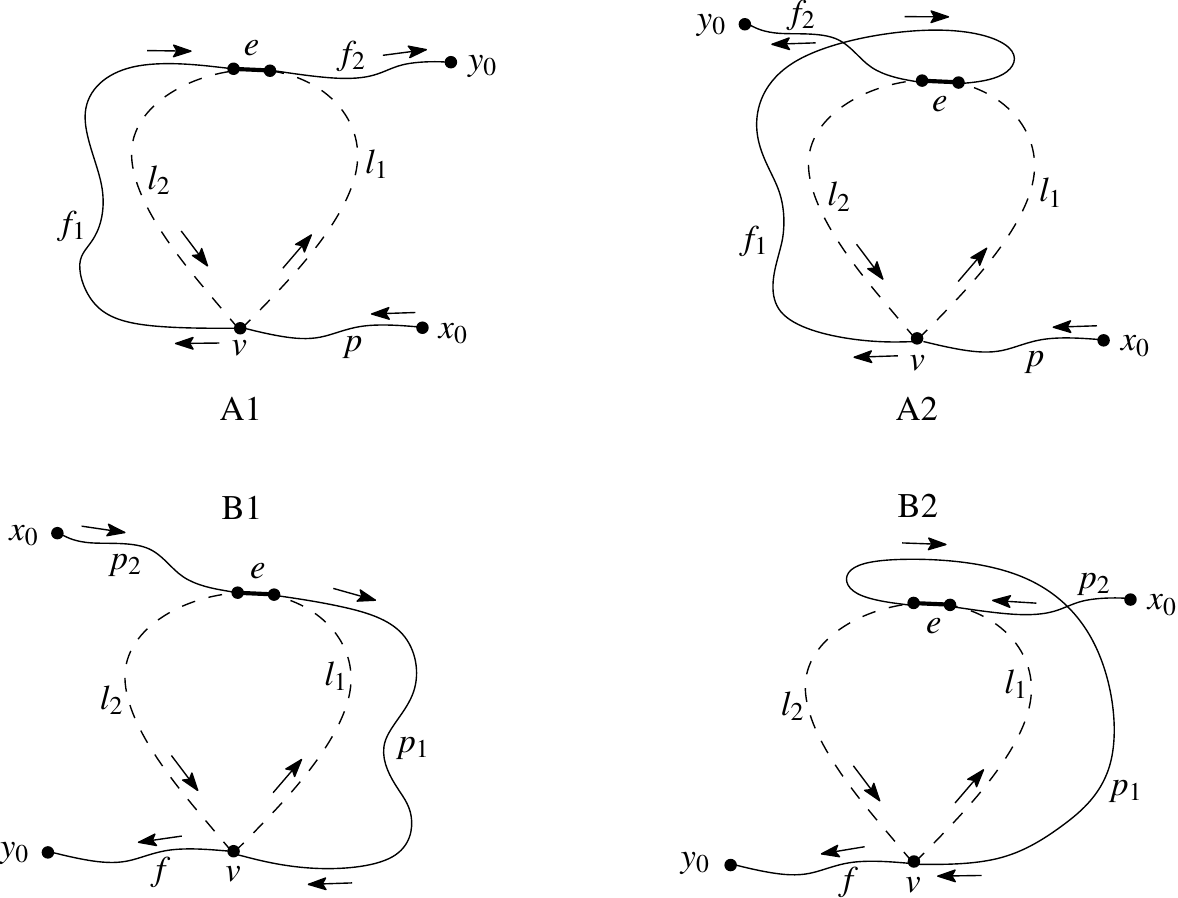}
\else\fi
\end{center}
\caption{Shortening admissible paths}
\label{fig:loops}
\end{figure}

\begin{proof}[Proof of Lemma \ref{lem:simple}]
We argue by contradiction. Assume that for some $x,y\in V(\widetilde G)$ there is a shortest $(x,y)$--admissible path $q$ in $G$, which
contains a loop on which there is an edge which is traversed more than once. Let $l$ be the shortest such loop, and denote by $v$ the vertex
through which the loop is entered and left. See Figure \ref{fig:loops}, the dashed loop. There are two possibilities: either there is an edge
on it traversed \emph{before} $q$ enters into $l$ (situations B1 and B2 in the figure), or \emph{after} it has left $l$ (situations A1 and A2).
In each of the possibilities, select the first edge $e$ in the respective direction which happens to be on $l$ (if there are such in both directions,
choose the closer one). Now in each case the edge $e$ is traversed either in the same direction as $l$ (situations A2 and B2) or in the opposite
direction (A1 and B1). Let us denote the various parts of the path $q$ as on the figures; we shall check each of the four possibilities and show
that we can find a shorter $(x,y)$--admissible path, which will be a contradiction. Let us remark that the parts of $q$
outside $l$ can cross $l$, or each other quite arbitrarily, it does not affect the argument.

In each of the cases, we show that there is a way to traverse all of $q$ (some parts maybe in the opposite direction) except that we
pass the edge $e$ exactly two times less. Also, we preserve the beginning and end vertices $x_0$ and $x_1$. It is clear from the
definition of admissibility that any new path like this will also be $(x,y)$--admissible. Let us now finally describe the traversing;
for brevity, we write ``opp $\dots$'' for ``$\dots$ in the opposite direction''.
\begin{itemize}
\item[A1:] We pass $p$, then opp $l_2$, opp $f_1$, $l_1$ and $f_2$.
\item[A2:] We pass $p$, $l_1$, opp $f_1$, opp $l_2$ and finally $f_2$.
\item[B1:] We pass $p_2$, $l_2$, $l_1$, $p_1$ and $f$.
\item[B2:] We pass $p_2$, then opp $l_1$, opp $p_1$, $l_2$ and finally $f$.
\end{itemize}
This finishes the proof.
\end{proof}

\section{Results}\label{sec:results}

In this section we prove three Propositions which yield the proof of Theorem~\ref{thm:main}.

Recall from the introduction that $\Gamma_0=\F_2$ is the free group of rank 2,  $\Gamma_n=\Gamma_{n-1}^{(2)}$ are defined
inductively for $n\geqslant 1$,
and $X_n$ denotes the Cayley graph of $\Gamma_0/\Gamma_n$ with respect to the image of the canonical generators of $\F_2$.

The {\it box space\/} of\, $\F_2$  associated with
the family $(\Gamma_n)_{n\geqslant 1}$ is, by definition,  the {\it coarse union\/} $(X,d_X)=\sqcup (X_n, d_{X_n})$ of $X_n$'s. That is,
$X=\sqcup X_n$ is the disjoint union with each $X_n$
endowed with its natural graph metric $d_{X_n}$, and  $d(X_n,X_m)\to \infty$ as $n+m\to\infty$.
For the sake of concreteness, let us declare that $d_X(X_n,X_m)=\diam(X_n)+\diam(X_m)+n+m$ whenever $n\not=m$.

\begin{lemma}\label{lem:trivial-intersection}
For the sequence $(\Gamma_n)_{n\geqslant1}$ described above, we have $\bigcap_{n\geqslant1}\Gamma_n=\{1\}$.
\end{lemma}

\begin{proof}
The  iterated squares are proper characteristic subgroups of the free group, hence,
by Levi's theorem~\cite[Ch.I, Prop.~3.3]{ls},  they have trivial intersection.
\end{proof}

\begin{proposition}
The metric space $(X,d_X)$ does not have property A.
\end{proposition}

\begin{proof}
Just use the previous Lemma, together with~\cite[Proposition 11.39]{roe:LOCG}.
\end{proof}

Note that $X_0$ is just the ``figure eight'' graph (Figure \ref{fig:H-graph}). It follows from the discussion in
Subsection \ref{subsec:z2} and Proposition \ref{prop:transitivity} that each $X_n$ is the $\Z/2$-homology cover of $X_{n-1}$, $n\geqslant 1$.
Hence we may endow each $X_n$ with the metric $d_{W_n}$, the wall metric that comes from $X_n$ being the $\Z/2$-homology cover of $X_{n-1}$.
Let us denote by $(X,d_W)=\sqcup (X_n,d_{W_n})$ the coarse union of $X_n$'s with the wall metrics.
As above, $d_W(X_n,X_m)=\diam_W(X_n)+\diam_W(X_m)+n+m$ whenever $n\not=m$ with $\diam_W$ being the corresponding diameter with respect to
the wall metric.

\begin{proposition}
The metric space $(X,d_W)$ is coarsely embeddable into  a Hilbert space.
\end{proposition}

\begin{proof}
For all finite sequences $x_1, \dots, x_r$ of elements of $X_n$ and
$\lambda_1, \dots, \lambda_r$ of real numbers such that
$\sum_{i=1}^{r}\lambda_i= 0,$ we have
$$\sum_{i,j}^{r}\lambda_i\lambda_j d_{W_n}(x_i,x_j) = \sum_{i,j}^{r}\lambda_i\lambda_j\sum_U \chi_U(x_i)(1-\chi_U(x_j))\leqslant 0,$$ where
$\chi_U$ is the characteristic function of the half space $U,$ and
$U$ ranges over all the half spaces in $X_n,$ see the proof of
\cite[Technical lemma]{NR}.

This and obvious equalities $d_{W_n}(x,y) = d_{W_n}(y,x),$
$d_{W_n}(x,x) = 0$ for all $x,y\in X_n$ mean that the wall metric
$d_{W_n}$ is a symmetric, normalized negative type kernel on each
$X_n,$ see ~\cite[Chapter 11]{roe:LOCG} for the terminology.

This fact, together with the definition of $d_W(X_n,X_m)$ above for
$n\not=m$, implies that, $d_W$ is an effective, symmetric,
normalized negative type kernel on $X$. The existence of such a
kernel is equivalent to the coarse embeddability by~\cite[Theorem
11.16]{roe:LOCG}.
\end{proof}

In fact, Hilbert space embeds coarsely into any  $\ell_p$ for $1\leqslant p\leqslant \infty$~\cite[Corollary 3.18]{willett}.
So, does the metric space $(X, d_W)$.

\begin{lemma}\label{lem:girth}
Let $\Lambda=\F(S)$ be the free group on a finite set of generators $S$. Let $(\Lambda_n)_{n\geqslant1}$ be a sequence
of finite index normal subgroups of $\Lambda$ satisfying $\Lambda_{n+1}\subset \Lambda_n$ and $\bigcap_{n\geqslant1}\Lambda_n=\{1\}$.
If we denote by $Y_n$ the Cayley graph of $\Lambda/\Lambda_n$ with respect to the image of the generating set $S$, then
$\girth(Y_n)\to\infty$ as $n\to\infty$.
\end{lemma}

\begin{proof}
The generating set $S$ determines a word metric on $\Lambda=\F(S)$,
which is also the graph metric on the Cayley graph of $\F(S)$ with respect to the generating set $S$;
this graph is a tree. Take $r>0$ and consider the ball $B_\Lambda(1,r)$ of radius $r$ around $1$ in $\F(S)$.
Since it is finite, the conditions that $\bigcap_{n\geqslant1}\Lambda_n=\{1\}$ and $\Lambda_{n+1}\subset\Lambda_n$
imply that there is some $n_r$, such that $B_\Lambda(1,2r)\cap \Lambda_n=\{1\}$ for all $n>n_r$.
That is, under the quotient map $\Lambda\twoheadrightarrow \Lambda/\Lambda_n$, the ball $B_\Lambda(1,r)$ bijectively
corresponds to the ball $B_{Y_n}(1,r)$ in the Cayley graph $Y_n$, for $n>n_r$. Since $Y_n$ is homogeneous, any ball of radius
$r$ in $Y_n$ is actually a tree, as is $B_\Lambda(1,r)$. It follows that $\girth(Y_n)>r$ for $n>n_r$ and we are done.
\end{proof}

\begin{proposition}
The identity map $\id:(X,d_X)\to (X,d_W)$ is a coarse equivalence (that is, both $\id$ and $\id^{-1}$ are coarse embeddings).
\end{proposition}

\begin{proof}
Since $\id:(X,d_X)\to (X,d_W)$ is $1$-Lipschitz (by Proposition \ref{prop:WlessG}),
it is sufficient to show the following statement: For any $R\geqslant0$ there exists $S\geqslant0$, such that $d_W(x,y)\leqslant R$
implies $d_X(x,y)\leqslant S$, for all $x,y\in X$.

So take $R\geqslant0$ and let $N\geqslant0$ be such that $\girth(X_n)>R$ for $n\geqslant N$;
this is possible by Lemma \ref{lem:girth}. By possibly enlarging $N$, we may further assume that $d_W(X_n,X_m)>R$ for any
$m,n\geqslant N$ with $m\not=n$, or for any $m<N$ and $n\geqslant N$.
Now define $S=\max\{R\}\cup\{d_X(x,y)\mid x,y\in \cup_{n<N}X_n\}$ (note that the latter is a finite set).

Let us check the required implication. Take $x,y\in X$ with $d_W(x,y)\leqslant R$. By our setup, there are
two possibilities: either both $x$ and $y$ belong to $\cup_{n<N}X_n$ and then $d_X(x,y)\leqslant S$ by the choice of $S$; or
$x,y\in X_n$ for some $n\geqslant N$. But for such $n$ the restriction of the identity map $\id:(X_n,d_{X_n})\to (X_n,d_{W_n})$
to any ball $B_X(z,R)$, $z\in X_n$, of radius $R<\girth(X_n)$, is an isometry onto the ball $B_W(z,R)$ by
Proposition \ref{prop:same-balls}. Thus necessarily $d_X(x,y)=d_W(x,y)\leqslant R\leqslant S$ and we are done.
\end{proof}

Observe that $(X,d_X)$ consists of Cayley graphs of groups with a fixed number of generators, it follows that
it has bounded geometry. Then the previous Proposition implies that $(X,d_W)$ has also bounded geometry.
This finishes the proof of our main result, Theorem~\ref{thm:main}.

\section{C*-algebraic context and Questions}\label{sec:apps}

In this section, we discuss our example in the C*-algebraic context. To every metric space, there are associated various
Roe C*-algebras, coarse--geometric analogues of group C*-algebras.

\begin{definition}[See \cite{roe:LOCG}]
Given a uniformly discrete metric space $X$ with bounded geometry, its \emph{algebraic uniform Roe algebra}
(sometimes also called \emph{Gromov's translation algebra}) $C^*_{\rm alg}X$ is the algebra of $X$-by-$X$
complex matrices $T=(t_{xy})_{x,y\in X}$, which have uniformly bounded entries (i.e.\ $\sup_{x,y}|t_{xy}|<\infty$) and finite
propagation (i.e.\ $\sup\{d(x,y)\mid x,y\in X\text{ and }t_{xy}\not=0\}<\infty$). This algebra is naturally represented
(by matrix multiplication) on $\ell^2X$; its closure in $\mathcal{B}(\ell^2X)$ is called the \emph{uniform Roe algebra} of $X$,
denoted $C^*_{\rm u}X$. The \emph{maximal uniform Roe algebra} of $X$, denoted $C^*_{\rm u,max}X$, is the closure of $C^*_{\rm alg}X$ in
the biggest C*-norm that it admits.
\end{definition}

\subsection{Relation between the uniform Roe algebra and its maximal version}

There is always a natural surjective map $\lambda:C^*_{\rm u,max}X\to C^*_{\rm u}X$. If $X$ has Property A, this map
is an isomorphism \cite{BO,spakula-willett}; the converse is an open question, a particular case of an open question about groupoid C*-algebras.
In the paper \cite{spakula-willett} the authors prove that if $X$ coarsely embeds into a Hilbert space, $\lambda$ induces
an isomorphism $\lambda_*:K_*(C^*_{\rm u,max}X)\to K_*(C^*_{\rm u}X)$. Thus our example from Theorem \ref{thm:box-space} is
covered by this result. This leads to an interesting question: Is the map $\lambda$ an isomorphism in this case?

\subsection{Nuclearity and exactness}

For a discrete group $\Gamma$ endowed with a proper metric, a result of Ozawa \cite{Ozawa} (see also Guentner--Kaminker \cite{GK}) says that
the following are equivalent:
\begin{itemize}
\item $\Gamma$ has Property A;
\item $\Gamma$ is C*-exact;
\item $C^*_{\rm u}|\Gamma|$ is nuclear.
\end{itemize}
For general bounded geometry metric spaces, the first and the third statements are also equivalent \cite{BO}.

It follows from this result that for $C^*_{\rm u}|\Gamma|$, exactness implies nuclearity (the converse is a general C*-algebraic fact).
It is an open question whether this is true for general metric spaces with bounded geometry. Progress in the positive direction has
been made by Brodzki, Niblo and Wright in \cite{BNW}. If one were to look for a counterexample, one should take a space without Property A,
so that its uniform Roe algebra is not nuclear, and prove that it is exact. Naturally, one can consider expanders, but it is known for a
large class that their uniform Roe algebras are not exact \cite{HG}.

From this perspective, one can ask whether the uniform Roe algebra of our example from Theorem \ref{thm:box-space} is exact or not.
The criterion from \cite{BNW} does not readily apply. However, it was pointed to us by Willett that an
elaboration of the argument in \cite[pages 96--98]{BO} shows that in fact uniform Roe algebras of box spaces in general are either nuclear or
not exact.

\bibliography{CE-not-A}
\bibliographystyle{plain}

\end{document}